\documentclass[11pt, twoside]{amsart}

\usepackage{amsmath,amssymb,amsthm,enumitem,stmaryrd,mathrsfs}

\usepackage{thm-restate,hyperref}

\usepackage[pdftex]{graphicx}

\usepackage[top=1.5in, bottom=1.5in, left=1.5in, right=1.5in]{geometry}

\usepackage{tikz}

\usetikzlibrary{arrows}

\allowdisplaybreaks

\theoremstyle{plain}
\newtheorem{theorem}{Theorem}[section]
\newtheorem{lemma}[theorem]{Lemma}
\newtheorem{cor}[theorem]{Corollary}
\newtheorem{prop}[theorem]{Proposition}
\newtheorem{conjecture}[theorem]{Conjecture}

\newtheorem*{lemma*}{Lemma}
\newtheorem*{cor*}{Corollary}
\newtheorem*{theorem*}{Theorem}

\theoremstyle{definition}

\newtheorem{problem*}{Problem}
\newtheorem{example}{Example}
\newtheorem{definition}[theorem]{Definition}

\theoremstyle{remark}

\newtheorem*{fact*}{Fact}
\newtheorem*{remark}{Remark}

\let\oldproofname=\proofname
\renewcommand{\proofname}{\rm\bf{\oldproofname}}

\newcommand{\R}{\mathbb R}
\newcommand{\Q}{\mathbb Q}
\newcommand{\Z}{\mathbb Z}

\newcommand{\F}{\mathbb F}

\newcommand{\phii}{\varphi}
\newcommand{\ep}{\varepsilon}

\makeatletter
\providecommand*{\twoheadrightarrowfill@}{%
  \arrowfill@\relbar\relbar\twoheadrightarrow
}
\providecommand*{\twoheadleftarrowfill@}{%
  \arrowfill@\twoheadleftarrow\relbar\relbar
}
\providecommand*{\xtwoheadrightarrow}[2][]{%
  \ext@arrow 0579\twoheadrightarrowfill@{#1}{#2}%
}
\providecommand*{\xtwoheadleftarrow}[2][]{%
  \ext@arrow 5097\twoheadleftarrowfill@{#1}{#2}%
}
\makeatother

\newcommand{\inj}{\hookrightarrow}

\newcommand{\op}[1]{\operatorname{{#1}}}

\newcommand{\mc}[1]{\mathcal{{#1}}}

\newcommand{\Hom}{\op{Hom}}

\newcommand{\Gal}{\op{Gal}}

\newcommand{\Ext}{\op{Ext}}

\newcommand{\GL}{\op{GL}}

\newcommand{\abs}[1]{\left\lvert#1\right\rvert}

\begin{document}
    \title{The Whitehead group and stably trivial $G$-smoothings}
    \author{Oliver H. Wang}
	\maketitle

% $R$ is a commutative ring. $G$ is a group. If I want it to be finite, I will say so. All modules will be left modules unless otherwise stated. For now, I'm letting $A$ denote a not necessarily commutative ring. I will let $I_n$ denote the $n\times n$ identity matrix. For the ``standard'' involution on $K$-theory groups, I will use the involution induced by $P\mapsto P^*$.
\begin{abstract}
A closed manifold $M$ of dimension at least $5$ has only finitely many smooth structures.
Moreover, smooth structures of $M$ are in bijection with smooth structures of $M\times\R$.
Both of these statements are false equivariantly.
In this paper, we use controlled $h$-cobordisms to construct infinitely many $G$-smoothings of a $G$-manifold $X$.
Moreover, these $G$-smoothings are isotopic after taking a product with $\R$.
\end{abstract}

\section{Introduction}
Let $G$ be a finite group.
A $G$-smoothing of a $G$-manifold $X$ consists of a pair $(Y,f)$ where $Y$ is a smooth $G$-manifold and $f:Y\rightarrow X$ is a $G$-homeomorphism.
If $Y$ is a smooth $G$-manifold, let $Y\times I$ denote the product smooth $G$-manifold where $G$ acts on $I$ trivially.
Two $G$-smoothings $(Y_i,f_i)$, $i=0,1$ are \emph{isotopic} if there is a $G$-homeomorphism $\alpha:Y_0\times I\rightarrow X\times I$ such that the following hold:
\begin{itemize}
\item $\alpha(-,t)$ is a $G$-homeomorphism $Y_0\times\{t\}\rightarrow X\times\{t\}$,
\item $\alpha(-,0)=f_0$ and
\item the composition $f_1^{-1}\circ\alpha(-,1):Y_0\rightarrow Y_1$ is a $G$-diffeomorphism.
\end{itemize}
In this paper, $G$-smoothings are considered up to isotopy.

As in classical smoothing theory, isotopy classes of $G$-smoothings can be classified by solutions to a lifting problem \cite{LashofRothenberg}.
However, unlike classical smoothing theory, closed $G$-manifolds may have infinitely many $G$-smoothings.
In \cite{SchultzExotic} and \cite{WangChern}, examples of closed $G$-manifolds with infinitely many $G$-smoothings are constructed by replacing the normal $G$-vector bundle of the fixed set with a non-isomorphic $G$-vector bundle.
In the current paper, we construct, for certain $G$-manifolds $X$, infinitely many non-isotopic $G$-smoothings whose fixed sets have the same normal bundle.
Rather than replacing the normal bundle of the fixed set, we replace a neighborhood of the unit sphere bundle of the normal bundle with an equivariant $h$-cobordism.

A key theorem in smoothing theory, proven by Kirby--Siebenmann, is the product structure theorem.
A smooth structure on $X$ gives a smooth structure on $X\times\R$.
The product structures theorem states that is a bijection when $X$ is a high dimensional manifold.
It is shown in \cite{WangChern} that an equivariant version of the stabilization map in the product structure theorem is not generally surjective.
Indeed, if $M$ is a $\Z/p$-manifold with a trivial action, then it has only finitely many $\Z/p$-smoothings.
But, if $H^2(M;\Q)\neq0$ and $2$ has odd order in $(\Z/p)^\times$, then $M\times(\R[\Z/p]/\R)^{\dim M}$ has infinitely many $\Z/p$-smoothings.
The $G$-smoothings in the present paper show that this assignment need not be injective.
If $X$ is a smooth $G$-manifold and $(Y,f)$ is a $G$-smoothing of $X$, then we say $(Y,f)$ is \emph{stably trivial} if there is a representation $\rho$ such that $f\times\op{id}:Y\times\rho\rightarrow X\times\rho$ is isotopic to the identity.

Our main theorem is the following.

\begin{restatable}{theorem}{MainTheorem}\label{thm: Main}
Let $G$ be an odd order cyclic group of order at least $5$.
Let $X$ be a smooth, compact, connected, semifree $G$-manifold and let $M$ be a component of the fixed point set.
Suppose the following conditions hold:
\begin{itemize}
\item $M$ is closed, aspherical and $\pi_1$-injective,
\item $\pi_1 M$ and $\pi_1 X$ satisfy the $K$-theoretic Farrell--Jones Conjecture and
\item Each component of $X^G$ has codimension at least $2$.
\end{itemize}
Then, there are infinitely many stably trivial $G$-smoothings of $X$ if either of the following hold:
\begin{enumerate}
\item $M$ (and, hence $X$) is odd dimensional.
\item $M$ is even dimensional, $H^2(M;\Q)\neq0$ and there are distinct prime factors $p_i,p_j$ of $\abs{G}$ such that $p_i$ has odd order in $(\Z/p_j)^\times$.
\end{enumerate}
\end{restatable}

We construct these $G$-smoothings from certain elements of the Whitehead group.
The $K$-theoretic Farrell--Jones conjecture for $M$ allows us to understand parts of the Whitehead group $\op{Wh}_1(\pi_1M\times G)$ by considering the homology of $M$ with coefficients in the lower $K$-theory of $\Z[G]$.
The $G$-smoothings in the first case of Theorem \ref{thm: Main} come from $H_0(M;\op{Wh}_1(G))$ whereas the $G$-smoothings in the second case come from $H_2(M;K_{-1}(\Z[G]))$.

\begin{remark}
An important subtlety in the definition of an isotopy is that we require $Y_0\times I$ to be the product smooth $G$-manifold.
Indeed, there are ways of giving the topological $G$-manifold $X\times I$ the structure of a smooth $G$-manifold so that it is not $G$-diffeomorphic to $Y_0\times I$ for any smooth $G$-manifold $Y_0$ \cite{BrowderHsiangProblem}.
This contrasts with the non-equivariant situation where the product smoothing gives a bijection between isotopy classes of smoothings on $X$ and isotopy classes of smoothings on $X\times I$ provided $\dim X\ge5$.
\end{remark}

\begin{remark}
Both the smoothings constructed in Theorem \ref{thm: Main} and those constructed in \cite{SchultzExotic} and \cite{WangChern} involve the second cohomology of the fixed point set and the order of elements in $(\Z/p)^\times$.
Though we believe this is coincidental, it would be very interesting if there were some deeper number theoretic or homotopy theoretic reason.
\end{remark}

We give some examples of $G$-manifolds where Theorem \ref{thm: Main} may be applied.

\begin{example}
When $G=\Z/p$, we may take $X=(M^{2n+1})^{\times p}$ with $G$ acting by permuting the coordinates.
By the first case of Theorem \ref{thm: Main}, this has infinitely many stably trivial $G$-smoothings.
\end{example}

\begin{example}
Let $G=\Z/m$ where $m$ is an integer with prime factors $p_i,p_j$ satisfying the conditions in the second case of Theorem \ref{thm: Main}.
Let $M$ be an even dimensional aspherical manifold such that $H^2(M;\Q)\neq0$ and $\pi_1M$ satisfies the $K$-theoretic Farrell--Jones conjecture.
Let $V$ be a free representation (i.e. $V^G=0$ and the only isotropy groups are $G$ and $0$) such that $\dim V>2$ and let $S^V$ denote the representation sphere.
Then the second case of Theorem \ref{thm: Main} shows that there are infinitely many stably trivial $G$-smoothings of $M\times S^V$, where $G$ acts trivially on $M$.
\end{example}

\subsection{Outline}
In Section \ref{section: Background}, we review some background.
In Section \ref{section: construction}, we describe the construction giving rise to the $G$-smoothings in Theorem \ref{thm: Main}.
This construction uses the fixed set of an involution on the Whitehead group of $\pi_1 M\times G$.
In Section \ref{section: control and assembly}, we analyze $K$-groups to show that, under the hypotheses of Theorem \ref{thm: Main}, there are infinitely many elements of the Whitehead group giving rise to the constructions of Section \ref{section: construction}.
In the appendix, we elaborate on Madsen--Rothenberg's analysis of the involution on $K_{-1}(\Z[G])$.

\subsection{Acknowledgments}
The author would like to thank Shmuel Weinberger for suggesting this project and for many helpful conversations.
This paper was partially written while the author was supported by NSF Grant DMS-1839968.

\section{Background}\label{section: Background}

\subsection{Whitehead Torsion}
Recall that, for a ring $R$, $K_1(R):=\GL(R)_{ab}$ and that the Whitehead group of a group $G$ is defined to be $\op{Wh}_1(G):=K_1(\Z[G])/\langle\pm g\rangle$.
There is an involution $\tau_1$ on $K_1(R[G])$ defined by sending a matrix $M$ to the inverse of its conjugate transpose.
This induces an involution on $\op{Wh}_1(G)$ which we also denote by $\tau_1$.

\begin{remark}
The involution $\tau_1$ is the negative of the involution considered in \cite{MilnorWhitehead}.
We will let $\tau_1$ be our ``standard'' involution as it behaves better with the involution on $K_0(R[G])$ defined by dualizing a projective module (see \ref{section: negative K theory involution}).
\end{remark}

Let $M_0$ be a closed, connected $n$-dimensional CAT-manifold where CAT is the category $TOP, PL$ or $DIFF$.
A cobordism over $M_0$ consists of a tuple $(W;M_0,M_1)$ where $W$ is an $(n+1)$-manifold with $\partial W=M_0\coprod -M_1$ where $-M_1$ denotes $M_1$ with a reversed orientation.
An $h$-cobordism is a cobordism such that the inclusion of each $M_i$ is a homotopy equivalence.
Two $h$-cobordisms $(W;M_0,M_1)$ and $(W';M_0,M_2)$ over $M_0$ are isomorphic if there is a CAT isomorphism $F:W_0\rightarrow W_1$ of manifolds with boundary which restricts to the identity on $M_0$.
When $n\ge5$, there is a bijection between isomorphism classes of $h$-cobordisms over $M_0$ and the Whitehead group given by Whitehead torsion $(W;M_0,M_1)\mapsto\tau(W,M_0)$.

The following formula can be found in \cite[Section 10]{MilnorWhitehead}.
\[
\tau(W,M_0)=(-1)^{n+1}\tau_1\cdot\tau(W,M_1)
\]
We will be interested in $h$-cobordisms where $M_0\cong M_1$, which are called \emph{inertial}.
A slightly more convenient class of $h$-cobordisms are the \emph{strongly inertial} $h$-cobordisms.
These are the inertial $h$-cobordisms such that the map $M_0\rightarrow M_1$ is homotopic to a homeomorphism.
The set of strongly inertial $h$-cobordisms forms a subgroup and it is a homotopy invariant of $M$.
Neither of these properties necessarily hold for inertial $h$-cobordisms.
Strongly inertial $h$-cobordisms are a finite index subgroup of the invariant subgroup $\op{Wh}_1(\pi_1 M)^{(-1)^{n+1}\tau_1}$.
This holds for any choice of CAT \cite[Proposition 5.2]{JahrenKwasik}.
We refer to \cite{JahrenKwasik} for more details on inertial and strongly inertial $h$-cobordisms.

The Whitehead group is $\pi_1\op{Wh}(G)$ for where $\op{Wh}(G)$ is a spectrum defined as follows.
For a space $X$, let $A^{-\infty}(X)$ denote the nonconnective $A$-theory spectrum of $X$.
Then $\op{Wh}(X)$ is defined to be the cofiber of the assembly $X_+\wedge A^{-\infty}(*)\rightarrow A^{-\infty}(X)$ and $\op{Wh}(G):=\op{Wh}(BG)$.

One may alternatively define a Whitehead spectrum using algebraic $K$-theory.
Let $\op{Wh}_{K}(X)$ be the cofiber of the assembly $B\pi_1X_+\wedge K(\Z)\rightarrow K^{-\infty}(\Z[\pi_1 X])$.
The linearization map $ A^{-\infty}(X)\rightarrow K^{-\infty}(\Z[\pi_1 X])$ is a map of spectra with involution \cite[Proposition 2.11]{VogellInvolution} and it induces isomorphisms of groups with involution
\[
\pi_n\op{Wh}(X)\rightarrow\pi_n\op{Wh}_{K}(X)
\]
for $n\le1$.
We may similarly take the Whitehead spectrum of $G$ to be $\op{Wh}_K(G):=\op{Wh}_K(BG)$.
For $n\le 1$, define $\op{Wh}_n(G):=\pi_n\op{Wh}(G)$.
Since we are only concerned with these homotopy groups, we will not differentiate between $\op{Wh}(G)$ and $\op{Wh}_K(G)$.

\subsection{Equivariant Homology and the Farrell--Jones Conjecture}
We will need Davis--L\"uck's equivariant homology and the Farrell--Jones conjecture.
We review the definitions and relevant results in the literature.

If $\Gamma$ is a group, let $\op{Or}(\Gamma)$ denote its orbit category.
Regarding an orbit $\Gamma/H$ as a discrete $\Gamma$-space gives a functor $i:\op{Or}(\Gamma)\rightarrow\Gamma-\op{Top}$ to the category of $\Gamma$-spaces.
If $\mathbf{E}:\op{Or}(\Gamma)\rightarrow Sp$ is a functor to the category of spectra and if $X$ is a $\Gamma$-space, we define the equivariant homology spectrum to be the left Kan extension
\[
H^\Gamma(X;\mathbf{E}):=\op{Lan}_{i}\mathbf{E}(X).
\]
The functor $H^\Gamma(-;\mathbf{E})$ is natural in $\mathbf{E}$.
If $\mathbf{E}$ is valued in spectra with involution then so is the functor $H^\Gamma(-;\mathbf{E})$.
If $\mathbf{E}'$ is another functor valued in spectra with involution and $f:\mathbf{E}\rightarrow\mathbf{E}'$ is a natural transformation respecting the involution, then the induced map $f_*:H^{\Gamma}(X;\mathbf{E})\rightarrow H^{\Gamma}(X;\mathbf{E}')$ is a map of spectra with involution.
These claims follow from the description of the Kan extension as a coend.

One functor we consider is the functor $\mathbf{K}:\op{Or}(\Gamma)\rightarrow Sp$ which satisfies the property that $\mathbf{K}(\Gamma/H)$ is the nonconnective $K$-theory spectrum $K^{-\infty}(\Z[H])$.
This is constructed thoroughly in \cite{DavisLuckEquivariant}.

\subsubsection{Classifying Spaces}
A family $\mc{F}$ of subgroups of $\Gamma$ is a set of subgroups which is closed under conjugacy and taking subgroups.
We will primarily be considering the family $\{1\}$ consisting of just the trivial subgroup and the family $\mc{FIN}$ consisting of the finite subgroups.
The family $\mc{VCY}$ of virtually cyclic subgroups is important in the statement of the Farrell--Jones conjecture.

Given a family of subgroups $\mc{F}$, the classifying space for $\mc{F}$ is denoted $E_{\mc{F}}\Gamma$ and is characterized by
\[
(E_{\mc{F}}\Gamma)^H\simeq\begin{cases}
*&H\in\mc{F}\\
\emptyset&H\notin\mc{F}
\end{cases}.
\]
In the case $\mc{F}=\mc{FIN}$, we write $\underline{E}\Gamma:=E_{\mc{FIN}}\Gamma$.

\begin{definition}\label{def: M(1 Fin)}
Let $\mc{F},\mc{G}$ be families of subgroups of $\Gamma$.
We say $\Gamma$ satisfies $(M_{\mc{F}\subseteq\mc{G}})$ if every subgroup $H\in\mc{G}\setminus\mc{F}$ is contained in a unique subgroup $H_{max}\in\mc{G}\setminus\mc{F}$ which is maximal in $\mc{G}\setminus\mc{F}$.
\end{definition}

Let $\mc{M}$ be a complete system of representatives of conjugacy classes of maximal finite subgroups of $\Gamma$.
L\"uck--Weiermann show that, for groups $\Gamma$ satisfying $(M_{\{1\}\subseteq\mc{FIN}})$, there is the following $\Gamma$-pushout diagram.
\[
\begin{tikzpicture}[scale=2]
\node (A) at (0,1) {$\coprod_{F\in\mc{M}}\Gamma\times_{N_{\Gamma}F}EN_{\Gamma}F$};\node (B) at (2,1) {$E\Gamma$};
\node (C) at (0,0) {$\coprod_{F\in\mc{M}}\Gamma\times_{N_{\Gamma}F}EW_{\Gamma}F$};\node (D) at (2,0) {$\underline{E}\Gamma$};
\path[->] (A) edge (B) (A) edge (C) (B) edge (D) (C) edge (D);
\end{tikzpicture}
\]
Taking the $\Gamma$-equivariant homology gives the following pushout diagram of spectra.
\[
\begin{tikzpicture}[scale=2]
\node (A) at (0,1) {$\bigvee_{F\in\mc{M}}H^{N_{\Gamma}F}_*(EN_{\Gamma}F;\mathbf{K})$};\node (B) at (2,1) {$H^{\Gamma}_*(E\Gamma;\mathbf{K})$};
\node (C) at (0,0) {$\bigvee_{F\in\mc{M}}H^{N_{\Gamma}F}_*(EW_{\Gamma}F;\mathbf{K})$};\node (D) at (2,0) {$H^{\Gamma}_*(\underline{E}\Gamma;\mathbf{K})$};
\path[->] (A) edge (B) (A) edge (C) (B) edge (D) (C) edge (D);
\end{tikzpicture}
\]

The $K$-theoretic Farrell--Jones Conjecture is the following statement.
\begin{conjecture}\label{conj: FJC}
The assembly map
\[
H^\Gamma(E_{\mc{VCY}}\Gamma;\mathbf{K})\rightarrow H^\Gamma(pt;\mathbf{K})= K^{-\infty}(\Z[\Gamma])
\]
is an equivalence.
\end{conjecture}

In order to simplify the diagram above rationally, we use the following proposition, which can be found in \cite[p. 746]{LuckReichSurvey}.
\begin{prop}\label{prop: FJC rational}
Suppose $\Gamma$ satisfies the $K$-theoretic Farrell--Jones conjecture.
Then, the assembly map
\[
H^{\Gamma}_m(\underline{E}\Gamma;\mathbf{K})\rightarrow H^{\Gamma}_m(pt;\mathbf{K})\cong K_m(\Z[\Gamma])
\]
is rationally an isomorphism.
\end{prop}

If $W_{\Gamma}F$ is torsion free, then $EW_{\Gamma}F\simeq\underline{E}N_{\Gamma}F$ as $N_{\Gamma}F$-spaces.
Under this hypothesis, Proposition \ref{prop: FJC rational} gives the following diagram, which is rationally a pushout.
\[
\begin{tikzpicture}[scale=2]
\node (A) at (0,1) {$\bigvee_{F\in\mc{M}}H_*(BN_{\Gamma}F;K(\Z))$}; \node (B) at (2,1) {$H_*(B\Gamma;K(\Z))$};
\node (C) at (0,0) {$\bigvee_{F\in\mc{M}}K_*(\Z[N_{\Gamma}F])$};\node (D) at (2,0) {$K_*(\Z[\Gamma])$};
\path[->] (A) edge (B) (A) edge (C) (B) edge (D) (C) edge (D);
\end{tikzpicture}
\]
Taking cofibers gives us a rational equivalence
\[
\bigvee_{F\in\mc{M}}\op{Wh}(N_{\Gamma}F)\rightarrow\op{Wh}(\Gamma).
\]
To summarize, we obtain the following.

\begin{prop}\label{prop: Wh rationally injective (alg)}
Suppose $\Gamma$ satisfies $(M_{\{1\}\subseteq\mc{FIN}})$ and that, for a maximal finite subgroup $F$, $W_{\Gamma}F$ is torsion free.
Then, the map
\[
\op{Wh}_m(N_{\Gamma}F)\rightarrow\op{Wh}_m(\Gamma)
\]
is rationally injective.
\end{prop}

In order to translate this algebraic statement into a topological statement, we need the following hypothesis (which is a specialization of \cite[Definition 4.49]{LuckTransformation} to the semifree case).

\begin{definition}\label{def: weak gap condition}
A semifree $G$-action on a manifold $X$ is said to satisfy the \emph{weak gap condition} if each component of the fixed set has codimension at least $3$.
\end{definition}

It appears to be well-known that the normalizers of finite subgroups of $\Gamma$ correspond to the fundamental groups of the lens space bundles of the fixed sets when $\pi_1 X$ is torsion free and when the action satisfies the weak gap condition.
However, we have not found a reference for this fact so we sketch a proof below.

\begin{lemma}\label{lem: normalizers}
Suppose a finite subgroup $G$ acts semifreely on a connected CW-complex $X$ and let $M$ be a component of the fixed set such that $\pi_1 M\rightarrow \pi_1 X$ is injective.
Let $\Gamma$ denote the semi-direct product $\pi_1 X\rtimes G$.
Then the subgroup $G=\{(0,g)\}\le\Gamma$ has normalizer $\pi_1 M\rtimes G\cong\pi_1 M\times G$.
If $\pi_1X$ is torsion free, then $G$ is a maximal finite subgroup of $\Gamma$.
\end{lemma}
\begin{proof}
Let $x_0\in M\subseteq X$ be a basepoint and let $\tilde{x}_0$ be a lift to the universal cover $\tilde{X}$.
Let $\tilde{M}\subseteq \tilde{X}$ denote the component of the preimage of $M$ containing the point $\tilde{x}_0$.
The subgroup $G=\{(0,g)\}\le\Gamma$ is precisely the stabilizer of $\tilde{M}$ under the action of $\Gamma$ on $\tilde{X}$ and the normalizer of $G$ is generated by $G$ and the subgroup of $\pi_1 X$ which sends $\tilde{M}$ to itself.
This is subgroup is $\pi_1 M$ which proves the first part of the proposition.

The second part is straightforward.
\end{proof}

\begin{lemma}\label{lem: pi_1 lens space bundle}
Suppose $E$ is the total space of a lens space bundle over a connected CW-complex $M$ obtained as the quotient of a sphere bundle $\tilde{E}$ by a free $G$-action.
Then,
\[
\pi_1 E=\pi_1 M\times G.
\]
\end{lemma}
\begin{proof}
There is a diagram
\[
\begin{tikzpicture}[scale=1.5]
\node (A) at (2,2) {$\pi_1\tilde{E}$};\node (B) at (0,1) {$G$};\node (C) at (2,1) {$\pi_1E$};\node (D) at (4,1) {$\pi_1M$}; \node (F) at (2,0) {$G$};
\path[->] (B) edge (C) (A) edge node[above right]{$\cong$} (D);
\path[->>] (C) edge node[above]{$\alpha$} (D) (C) edge node[left]{$\beta$} (F);
\path[right hook->] (A) edge (C);
\end{tikzpicture}
\]
from which one sees that the composite $G\rightarrow G$ is surjective, and hence an isomorphism.
Then the function $(\alpha,\beta):\pi_1 E\rightarrow \pi_1 M\times G$ is an isomorphism.
\end{proof}

Suppose $G$ acts smoothly and semifreely on a manifold $X$ such that $\pi_1X$ is torsion free and such that the action satisfies the weak gap condition.
Let $M$ be a $\pi_1$-injective component of the fixed set and let $\nu$ denote the normal bundle.
Let $X'$ denote the $G$-manifold obtained from $X$ by removing an equivariant neighborhood of the fixed set.
Then $\pi_1X'/G=\Gamma$ and one can check that the inclusion of the lens space bundle
\[
i:S\nu/G\rightarrow X'/G
\]
induces the inclusion of the normalizer
\[
N_{\Gamma}G\rightarrow \Gamma.
\]

Applying Proposition \ref{prop: Wh rationally injective (alg)}, we obtain the following.
\begin{prop}\label{prop: Wh rationally injective (top)}
With the notation and assumptions above,
\[
i_*:\op{Wh}_m(S\nu/G)\rightarrow\op{Wh}_m(X'/G)
\]
is rationally injective.
\end{prop}

\subsection{Controlled $h$-Cobordisms}
We will be interested in $h$-cobordisms of lens space bundles over a manifold $M$.
In order to study such $h$-cobordisms, it is helpful to use the notion of control introduced by Quinn \cite{QuinnEnds2}.
In our applications, our objects will be controlled over a compact manifold so our exposition here is slightly simpler than what is discussed in \cite{QuinnEnds2}.

\begin{definition}\label{def: controlled homotopy}
Let $(M,d)$ be a compact metric space and let $\ep>0$.
Suppose $p:E\rightarrow M$ and $p':E'\rightarrow M$ are proper maps.
\begin{enumerate}
\item A function $f:E\rightarrow E'$ is \emph{$\ep$-controlled} if, for all $x\in E$, $d(p(x),p'\circ f(x))<\ep$.
\item A homotopy $H:E\times I\rightarrow E'$ is \emph{$\ep$-controlled} if, for all $x\in E$, the set $p'\circ H(x,I)$ has diameter less than $\ep$.
\end{enumerate}
\end{definition}

\begin{remark}
If $p:E\rightarrow M$ and $p':E'\rightarrow M$ are fiber bundles over $M$, then any map of bundles is controlled for all $\ep>0$.
If $E$ and $E'$ are isomorphic CAT block bundles over $M$, then for each $\ep>0$, there is an $\ep$ controlled CAT isomorphism $E\to E'$.
\end{remark}

\begin{definition}\label{def: controlled h-cobordism}
Let $(W;E,E')$ be an $h$-cobordism and let $p:W\rightarrow M$ be a proper map.
We say that $(W;E,E')$ is a \emph{controlled $h$-cobordism with respect to $p$} if, for all $\ep>0$, there is a deformation retraction of $W$ to $E$ which is $\ep$-controlled. 

Two controlled $h$-cobordisms $\phii_i:(W_i;E_i,E_i')\rightarrow M$, $i=0,1$, are \emph{controlled isomorphic} if, for all $\ep>0$, there is an isomorphism of $h$-cobordisms $F:W_0\rightarrow W_1$ which is $\ep$-controlled over $M$.
\end{definition}

If $(W_0;E_0,E_0')$ is a controlled $h$-cobordism, there is a controlled $h$-cobordism $(W_1;E_0',E_1)$ such that $(W_0\cup_{E_0'}W_1;E_0,E_1)$ is controlled isomorphic to a product (see \cite[Theorem 1.2]{QuinnEnds2} and \cite[Proposition 1.7]{QuinnEnds2}).

\begin{prop}\label{prop: swindle}
Suppose $\xi\rightarrow M$ is a $G$-vector bundle whose fibers are free $G$-representations.
Let $S\xi$ denote the sphere bundle of $\xi$ and let $p:E\rightarrow M$ denote the lens space bundle obtained by quotienting.
Let $(W;E,E)$ be a controlled $h$-cobordism with respect to $p$ and let $\tilde{W}$ denote the $G$-cover.
Then there is a $G$-homeomorphism $\Phi:\tilde{W}\cup_{S\xi}D\xi\rightarrow D\xi$ where $D\xi$ denotes the disk bundle.
If $f:S\xi\rightarrow S\xi$ is a $G$-homeomorphism, then we may assume the homeomorphism $\Phi$ restricts to $f$ on the boundary.
\end{prop}
\begin{proof}
Let $\ep_n$ be a sequence such that $\sum\ep_n<\infty$.
Write $(W_0;E_0,E_1):=(W;E,E)$ and let $(W_1;E_1,E_2)$ denote a controlled $h$-cobordism such that $(W_0\cup W_1;E_0,E_2)$ is controlled isomorphic to $(E\times I;E,E)$.
Let $F_1:W_0\cup W_1\rightarrow E\times I$ be an $\ep_1$-controlled isomorphism and let $f_1$ denote the restriction of $F_1$ on $E_2$.
Inductively, define
\begin{itemize}
\item $(W_n;E_n,E_{n+1})$ to be a controlled $h$-cobordism such that $(W_{n-1}\cup_{f_{n-1}}W_n;E_{n-1},E_{n+1})$ is controlled isomorphic to $(E\times I;E,E)$,
\item $F_n:(W_{n-1}\cup_{f_{n-1}}W_n;E_{n-1},E_{n+1})\rightarrow (E\times I;E,E)$ to be a an $\ep_n$-controlled isomorphism and
\item $f_n$ to be the restriction of $F_n$ on $E_{n+1}$.
\end{itemize}
All $E_n$ are of course diffeomorphic to $E$.

Define
\[
Y:=W_0\cup W_1\cup_{f_1} W_2\cup_{f_2}W_3\cup\cdots.
\]
Clearly, $Y$ is homotopy equivalent to $E$ so we may take a $G$-cover $\tilde{Y}$.
Define $p_Y:Y\rightarrow M$ as follows.
For $x\in W_n\setminus E_{n+1}$, let $p_Y(x)$ be the image of $x$ under $p:W_n\to M$ where the first map comes from an $\ep_n$-deformation retraction.
Note that $p_Y$ is not, in general, continuous.

Topologize $\tilde{Y}\cup M$ by declaring that a sequence of points $x_n\in W_{k_n}$ converges to $m\in M$ if $p_Y(x_n)$ converges to $m$ and if $k_n\rightarrow\infty$.
Let $F:Y\rightarrow E\times[0,\infty)$ be defined to be $F_{2n+1}$ on $W_{2n}\cup_{f_{2n}}W_{2n+1}$ and let $G:Y\rightarrow W\cup_{E} E\times[0,\infty)$ be defined to be the identity $W_0\rightarrow W$ and $F_{2n}$ on $W_{2n-1}\cup_{f_{2n-1}}W_{2n}$.
Then $\tilde{F}$ and $\tilde{G}$ are equivariant homeomorphisms
\[
\tilde{W}\cup_{S\xi}S\xi\times[0,\infty)\xleftarrow{\tilde{G}}\tilde{Y}\xrightarrow{\tilde{F}}S\xi\times[0,\infty)
\]
which extends to equivariant homeomorphisms
\[
\tilde{W}\cup_{S\xi}D\xi\leftarrow \tilde{Y}\cup M\rightarrow D\xi.
\]
Taking $\Phi:\tilde{W}\cup_{S\xi}D\xi\rightarrow D\xi$ finishes the proof.
\end{proof}

\begin{figure}
\includegraphics[scale=0.65]{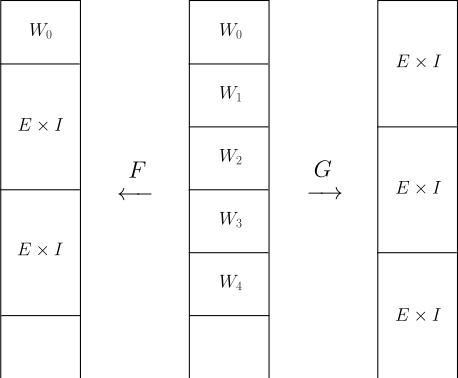}
\label{fig: F and G}
\caption{$F$ and $G$ in the proof of Proposition \ref{prop: swindle}}
\end{figure}

In Section \ref{section: control and assembly}, we discuss the relationship between the assembly map and controlled $h$-cobordisms.

\section{The Construction of Smoothings}\label{section: construction}

Suppose $X$ is a smooth, semifree $G$-manifold and let $M$ be a component of $X^G$.
Let $\nu$ denote the normal bundle of $M$ and let $\mathring{D}\nu$ denote the interior of the disk bundle $D\nu$.
Then $S\nu$ has a free $G$-action and $E:=S\nu/G$ is a lens space bundle over $M$.
Define $X':=X\setminus \mathring{D}\nu$.

Let $(W;E,E)$ be a smooth inertial $h$-cobordism controlled over $M$ and let $\tilde{W}$ be the $G$-cover.
Define
\[
X_W:=X'\cup\tilde{W}\cup D\nu.
\]
By Proposition \ref{prop: swindle}, there is an equivariant homeomorphism $f_W:X_W\rightarrow X$.
The equivariant smooth structures we study will be of the form $(X_W,f_W)$.

We record the following.

\begin{prop}\label{prop: stabilization}
The $G$-smoothing $f_W\times\op{id}:X_W\times\R\rightarrow X\times\R$ is isotopic to the identity.
\end{prop}

\begin{proof}
Let $(W;E_0,E_1)$ be a controlled $h$-cobordism.
Since the Euler characteristic of $S^1$ vanishes, there is an isomorphism
\[
F:W\times S^1\xrightarrow{\cong} E_0\times I\times S^1
\]
of $h$-cobordisms controlled over $M$ (see \cite[Proposition 1.7]{QuinnEnds2}).
Taking the $\Z$-cover shows that $W\times\R\cong E_0\times I\times\R$.
The proposition follows from the construction of $(X_W,f_W)$.
\end{proof}

Our goal in the remainder of this section is to show that, under certain hypotheses, different choices of $h$-cobordisms yield different $G$-smoothings.

\subsection{An Alternate Interpretation of the Whitehead Group}
Let $A$ be a finite complex.
The Whitehead group $\op{Wh}_1(A)$ of $A$ may be defined as follows.
An element is represented by a pair $(X,A)$ where the inclusion $A\hookrightarrow X$ is a homotopy equivalence.
Two pairs $(X,A)$ and $(Y,A)$ are equivalent if $Y$ can be obtained from $X$ by a series of elementary expansions and collapses.
The sum $(X,A)+(Y,A)$ is given by $(X\cup_A Y,A)$ and the identity is $(A,A)$.
A continuous function $f:A\rightarrow B$ induces a map on Whitehead groups as follows.
\[
f_*(X,A)=(X\cup_A\op{Cyl}(f),B)
\]
When $A$ is connected, this is isomorphic to $\op{Wh}_1(\pi_1 A)$.

If $f:B\rightarrow A$ is a homotopy equivalence, then the pair $(\op{Cyl}(f),A)$ is the torsion of $f$.
If $A_0$ is a compact manifold (possibly with boundary), an $h$-cobordism $(W;A_0,A_1)$ determines an element in the Whitehead group $\op{Wh}_1(A_0)$ this way via the homotopy equivalence $A_1\rightarrow A_0$.
Using this interpretation of the Whitehead group, the following can be verified.

\begin{lemma}\label{lem: h-cobordism sum}
Let $A_0$ and $B_0$ be compact manifolds with boundary and let $(W;A_0,A_1)$ and $(V;B_0,B_1)$ be $h$-cobordisms of manifolds with boundary.
Let $\partial_0 A$ be a component of $\partial A_0$ which is homeomorphic to a component of $\partial B_0$.
Let $i_{A_0}:A_0\inj A_0\cup_{\partial_0 A}B_0$ and $i_{B_0}:B_0\inj A_0\cup_{\partial_0 A}B_0$ be the inclusions.
Then
\[
(W\cup_{\partial_0 A\times I}V;A_0\cup_{\partial_0 A}B_0,A_1\cup_{\partial_0 A}B_1)
\]
is an $h$-cobordism and
\[
\tau(W\cup_{\partial_0 A\times I}V)=(i_{A_0})_*\tau(W)+(i_{B_0})_*\tau(V)\in\op{Wh}_1(A_0\cup_{\partial_0 A}B_0).
\]
\end{lemma}

\subsection{Distinguishing Smooth Structures}
\begin{prop}\label{prop: distinguishing smooth structures}
Suppose $X$, $G$ and $M$ are as in the hypotheses of Proposition \ref{prop: Wh rationally injective (top)}.
Let $W_0$ and $W_1$ be controlled $h$-cobordisms as in Section \ref{section: construction}.
If $\tau(W_0)\neq\tau(W_1)$ in $\op{Wh}_1(\pi_1 M)\otimes\Q$, then $(X_{W_0},f_{W_0})$ and $(X_{W_1},f_{W_1})$ are not isotopic $G$-smoothings.
\end{prop}

\begin{proof}
To ease notation, we assume $M$ is the only component of the fixed set.

Suppose otherwise.
Then there is a smooth $G$-manifold $V$, a $G$-homeomorphism $\alpha:V\rightarrow X\times I$ and $G$-diffeomorphisms
\[
d_i:X_{W_i}\rightarrow\partial_i V
\]
satisfying $(\alpha|_{\partial_i V})\circ d_i=f_{W_i}$ where $\partial_i V=\alpha^{-1}(X\times\{i\})$.

We decompose $V$ into submanifolds with boundary as follows.

By abuse of notation, write $M\times I$ for the preimage $\alpha^{-1}(M\times I)$.
Let $\nu$ be the normal bundle of $M$.
Remove the normal bundle of $M\times I$ to obtain a smooth $G$-manifold $V'$ with boundary
\[
\partial V'=(X'\cup_{S\nu}\tilde{W}_0)\cup(S\nu\times I)\cup(X'\cup_{S\nu}\tilde{W}_1).
\]
The $G$-action on $V'$ is free and $V'/G$ is an $h$-cobordism of manifolds with boundary.

Now, let $Z:=\alpha^{-1}(\alpha\circ d_0(S\nu)\times I)$ where $S\nu=\partial X'$ is where $\tilde{W}_0$ is attached.
Note that $Z\cap (X'\cup_{S\nu}\tilde{W}_1)=S\nu$, the submanifold where $\tilde{W}_1$ is attached to $X'$.
Let $\hat{W}\subseteq V'$ denote the submanifold bounded by $Z,\tilde{W}_0,\tilde{W}_1$ and $S\nu\times I$.
The complement of $\hat{W}$ is homeomorphic to $X'\times I$.

Note that $Z$ is $G$-homeomorphic to $S\nu\times I$ and $\hat{W}/G$ is an $h$-cobordism of the manifolds with boundary $W_0$ and $W_1$.
Since $\tau(W_0)\neq\tau(W_1)$, $\hat{W}/G$ cannot be a trivial $h$-cobordism so $\tau(\hat{W}/G)\neq0$.
Applying Lemma \ref{lem: h-cobordism sum} and Proposition \ref{prop: Wh rationally injective (top)}, we see that $V'/G$ is a nontrivial $h$-cobordism of manifolds with boundary.

This shows that the smooth $G$-manifold $V$ is a nontrivial isovariant $h$-cobordism (see \cite[4.D]{LuckTransformation}).
Under our hypotheses, the weak gap condition \cite[4.49]{LuckTransformation} is satisfied so the isovariant Whitehead group injects into the equivariant Whitehead group.
Therefore, $V$ is not equivariantly diffeomorphic to a product $X_{W_0}\times I$.
\end{proof}
\begin{figure}
\includegraphics[scale=0.65]{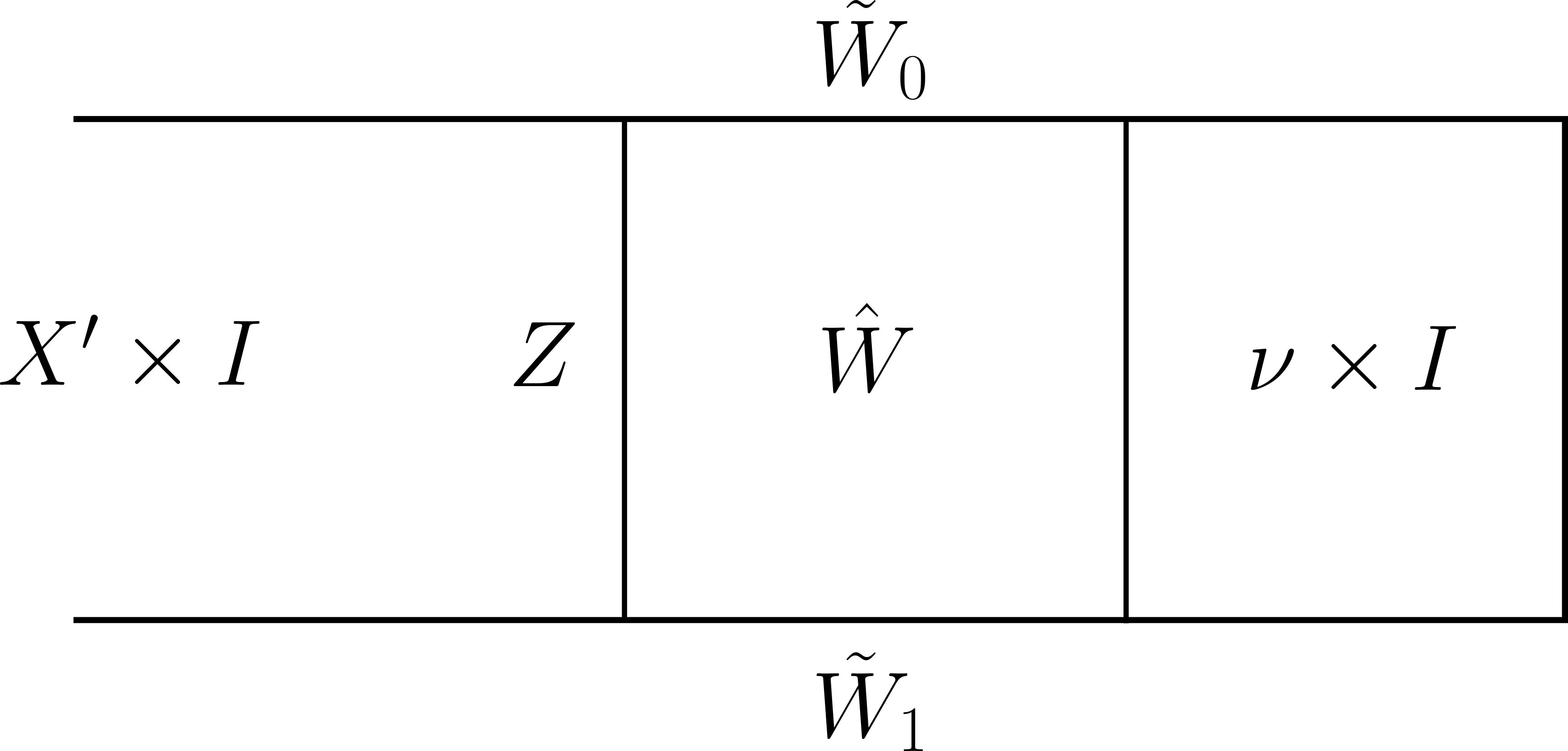}
\label{fig: V}
\caption{$V$ in the proof of Proposition \ref{prop: distinguishing smooth structures}}
\end{figure}

\section{Control and Assembly}\label{section: control and assembly}
In this section, we use the assembly map and a result of Quinn to realize certain elements of the Whitehead group as the torsion of controlled, inertial $h$-cobordisms.
The ideas here have also been studied by Steinberger--West \cite{SteinbergerWest} and Steinberger \cite{Steinberger}.

\subsection{Controlled $h$-Cobordisms and Homology}
Let $p:E\rightarrow M$ be a bundle with connected fiber $F$ and suppose $M$ is connected.
Denote $\pi:=\pi_1 M$.
Following \cite{FarrellLuckSteimle}, define a functor $\underline{E}:Or(\pi)\rightarrow Top$ by sending each orbit $\pi/H$ to the pullback bundle over the cover of $M$ corresponding to $H$.
Let $\mathbf{E}:Top\rightarrow Sp$ be a functor from spaces to spectra.
Define $\mathbf{E}(p)$ to be the composite $\mathbf{E}\circ\underline{E}$.
For a $\pi$-CW-complex $X$, we may define the Davis--L\"uck equivariant homology groups $H^{\pi}_*\left(X;\mathbf{E}(p)\right)$.
We are primarily interested in the case $\mathbf{E}$ is the Whitehead spectrum $\op{Wh}$.

In \cite{QuinnEnds2}, Quinn defines homology with coefficients in a spectrum valued functor $\mathbf{E}:Top\rightarrow Sp$.
Let $\mathbb{H}(M;\mathbf{E})$ denote this homology spectrum and let $\mathbb{H}_k(M;\mathbf{E})$ denote the homotopy groups.
He shows that a particular homology group $\mathbb{H}_1(M;\mc{S}(p))$ is in bijection with $h$-cobordisms $(W;E,E')$ controlled over $M$ where $p:E\rightarrow M$.
Farrell--L\"uck--Steimle compare Quinn's homology group with the Davis--L\"uck equivariant homology theory.

\begin{prop}\label{prop: classifying controlled h-cobordism}
Suppose $M$ is an aspherical manifold and $E$ is a closed manifold.
Let $\tilde{M}$ be the universal cover of $M$ and let $\pi=\pi_1 M$.
Let $p:E\rightarrow M$ be a bundle with connected fiber $F$ and let $\phii:(W;E,E')\rightarrow M$ be a controlled $h$-cobordism.
There is an invariant $q(\phii,p)\in H_1^{\pi}(\tilde{M};\op{Wh}(p))$ such that the following hold.
\begin{enumerate}
\item Two controlled $h$-cobordisms are controlled isomorphic if and only if their invariants are equal.
\item When $\dim E\ge5$, all invariants in this group can be realized.
\end{enumerate}
\end{prop}
\begin{proof}
This follows from \cite[1.2]{QuinnEnds2} and the identification of Quinn's homology group with $H_1^{\pi}(\tilde{M};\op{Wh}(p))$ in \cite[Lemma 4.9]{FarrellLuckSteimle}.
\end{proof}

\subsection{Assembly}

Quinn also defines an assembly map $\mathbb{H}_1(M;\mc{S}(p))\rightarrow \op{Wh}(\pi_1 E)$ which can be compared to the Farrell--Jones assembly in the Davis--L\"uck formulation.
Geometrically, Quinn's assembly sends a controlled $h$-cobordism $(W;E,E')$ to the torsion $\tau(W,E)$ where we consider $(W;E,E')$ as an ``uncontrolled'' $h$-cobordism.
Farrell--L\"uck--Steimle show that, when $M$ is aspherical, the Quinn assembly map has the same image as the Davis--L\"uck assembly map \cite[Lemma 4.9.iii]{FarrellLuckSteimle}.
Finally, they show that the Davis--L\"uck assembly map
\[
H_1^{\pi}(\tilde{M};\op{Wh}(p))\rightarrow H_1^{\pi}(pt;\op{Wh}(p))=\pi_1(\op{Wh}(E))
\]
is split injective provided $M$ is aspherical, $p:E\rightarrow M$ is $\pi_1$-surjective and $\pi$ satisfies the $K$-theoretic Farrell--Jones conjecture.

\subsection{Some Additional Simplifications}
Returning to our geometric situation, we have a closed aspherical $n$-manifold $M$ whose fundamental group $\pi$ satisfies the $K$-theoretic Farrell--Jones conjecture.
Moreover, the map $p:E\rightarrow M$ is a lens space bundle with fiber $F$.
The only orbits involved in the construction of the Davis--L\"uck homology spectrum is the orbit $G/pt$.
Since $\op{Wh}(p)(G/pt)=\op{Wh}(F)$, there is an isomorphism $H_1^{\pi}(\tilde{M};\op{Wh}(p))\cong H_1(M;\op{Wh}(F))$ where the right hand side is a twisted generalized homology group.

We may simplify this further.
Recalling that $\pi_1E\cong G\times\pi$, we see that the action of $\pi$ on the fundamental group $\pi_1 F$ is trivial.
Linearization gives an isomorphism
\[
H_1(M;\op{Wh}(F))\rightarrow H_1(M;\op{Wh}_{K}(F))
\]
of twisted generalized homology groups.
But since the action of $\pi$ on $\op{Wh}_{K}(F)$ is determined entirely by its action on $\pi_1 F$, the homology group on the right hand side is untwisted.

The following proposition follows from Proposition \ref{prop: distinguishing smooth structures}, Proposition \ref{prop: classifying controlled h-cobordism} and the above discussion.

\begin{prop}\label{prop: equiv homology suffices}
Each element of $H_1(M;\op{Wh}_{K}(F))^{(-1)^{n+1}\tau_1}$ gives a unique $G$-smoothing.
Here, the homology group is untwisted.
\end{prop}

\subsection{Involutions on $H_1\left(M;\op{Wh}_{K}(F)\right)$}
We now reduce the study of the involution $\tau_1$ on $H_1\left(M;\op{Wh}_{K}(F)\right)$ to the study of the involution on $K_{-1}(\Z[G])$.

\begin{prop}\label{prop: H_1(M;Wh(G)) rational computation}
Suppose $X$ is a CW complex.
Then
\[
H_1(X;\op{Wh}_{K}(F))_{(0)}\cong H_0(X;\op{Wh}(G))_{(0)}\oplus H_2(X;K_{-1}(\Z[G]))_{(0)}.
\]
\end{prop}
\begin{proof}
Since we are only interested in the first homology group, the Atiyah-Hirzebruch spectral sequence is easy to analyze.
Its $E^2$-page is
\[
\begin{tikzpicture}[scale=1.5]
\node (A) at (0,2) {$H_0(X;\op{Wh}(G))$};\node (B) at (3,2) {$H_1(X;\op{Wh}(G))$};\node (C) at (6,2) {$H_2(X;\op{Wh}(G))$};
\node (D) at (0,1) {$H_0(X;\tilde{K}_0(\Z[G]))$};\node (E) at (3,1) {$H_1(X;\tilde{K}_0(\Z[G]))$};\node (F) at (6,1) {$H_2(X;\tilde{K}_0(\Z[G]))$};
\node (G) at (0,0) {$H_0(X;K_{-1}(\Z[G]))$};\node (H) at (3,0) {$H_1(X;K_{-1}(\Z[G]))$};\node (I) at (6,0) {$H_2(X;K_{-1}(\Z[G]))$};
\end{tikzpicture}
\]
but the left column splits off, $\tilde{K}_0(\Z[G])$ is finite and Carter's vanishing theorem implies that there are no lower rows.
Therefore, $E^{\infty}_{0,1}=E^2_{0,1}\cong\op{Wh}_1(G)$, $E^{\infty}_{1,0}$ is a finite group and $E^{\infty}_{2,-1}=E^{2}_{2,-1}\cong H_2(X;K_{-1}(\Z[G]))$.
\end{proof}

We would like to endow the right hand side of the expression in Proposition \ref{prop: H_1(M;Wh(G)) rational computation} with an involution such that the decomposition of $H_1(X;\op{Wh}_{K}(F))_{(0)}$ above respects the involution.
On $H_0(X;\op{Wh}_1(G))$, the involution is just given by $\tau_1$ on $\op{Wh}_1(G)$.
The map $H_0(X;\op{Wh}_1(G))\rightarrow H_1(X;\op{Wh}_{K}(F))$ respects the involution since it is induced by the inclusion of a point.

We show there is an involution on $H_2(X;K_{-1}(\Z[G]))$ and a quotient map $H_1(X;\op{Wh}_{K}(F))\rightarrow H_2(X;K_{-1}(\Z[G]))$ respecting the involution.
We do this by considering the filtration of the left hand side.
Recall that Atiyah--Hirzebruch spectral sequence is given by a filtration arising from skeleta of $X$.
If $X^{(i)}$ denotes the $i$-skeleton, then the filtration on $H_1(X;\op{Wh}_{K}(F))$ is given by
\[
F_0\subseteq F_1\subseteq F_2\subseteq F_3\subseteq\cdots\subseteq H_1(M;\op{Wh}_{K}(F))
\]
where $F_i=\op{im}(H_1(X^{(i)};\op{Wh}_{K}(F))\rightarrow H_1(X;\op{Wh}_{K}(F)))$ and $E^{\infty}_{i,1-i}=F_i/F_{i-1}$.
In particular, $F_i/F_{i-1}=0$ for $i\ge3$.
This implies $F_2=F_3=\cdots=H_1(X;\op{Wh}_{K}(F))$.
So
\begin{equation}\label{eq: H_2(X;K_-1) identification}
H_2(X;K_{-1}(\Z[G]))\cong H_1(X;\op{Wh}_{K}(F))/H_1(X^{(1)};\op{Wh}_{K}(F)).
\end{equation}
The following proposition becomes immediate.

\begin{prop}\label{prop: H_1(M;Wh) integral computation}
If $X\rightarrow Y$ is a map of CW complexes then there is a commuting diagram of abelian groups with involution
\[
\begin{tikzpicture}[scale=2]
\node (A) at (0,1) {$H_0(X;\op{Wh}_1(G))$};\node (B) at (2,1) {$H_1(X;\op{Wh}_{K}(F))$};\node (C) at (4,1) {$H_2(X;K_{-1}(\Z[G]))$};
\node (D) at (0,0) {$H_0(Y;\op{Wh}_1(G))$};\node (E) at (2,0) {$H_1(Y;\op{Wh}_{K}(F))$};\node (F) at (4,0) {$H_2(Y;K_{-1}(\Z[G]))$};
\path[->] (A) edge (B) (B) edge (C) (A) edge (D) (B) edge (E) (C) edge (F) (D) edge (E) (E) edge (F);
\end{tikzpicture}
\]
where the left horizontal maps are injective, the right horizontal maps are surjective, the horizontal composites are trivial and the rows are exact after rationalizing.
\end{prop}

Note that the involution on $H_0(X;\op{Wh}_1(G))$ is given by its identification with $H_1(\pi_0 X;\op{Wh}_{K}(F))$.
So, understanding the involution on this homology group amounts to understanding the involution on the spectrum $\op{Wh}_{K}(F)$.
The involution on the group $H_2(X;K_{-1}(\Z[G]))$ is defined by the identification (\ref{eq: H_2(X;K_-1) identification}) above.
To compute the involution, we reduce to the case where $X$ is a surface by noting that every element of $H_2(X;\Z)$ is of the form $f_*[\Sigma_g]$ where $f:\Sigma_g\rightarrow M$ is a map from a closed oriented surface.
Moreover, every closed oriented surface admits a map to $T^2$ which is an isomorphism on $H_2$.
By considering these maps, Proposition \ref{prop: H_1(M;Wh) integral computation} gives the following result.

\begin{prop}\label{prop: H_2(X;K)}
Suppose $H_2(X;\Z)$ is a finitely generated group of rank $r$.
There is a map of abelian groups with involution
\[
H_2(T^2;K_{-1}(\Z[G]))^{r}\rightarrow H_2(X;K_{-1}(\Z[G]))
\]
which is an isomorphism when restricted to the torsion free part.
\end{prop}

\begin{remark}
In the statement of Proposition \ref{prop: H_2(X;K)}, we are implicitly using that $K_{-1}(\Z[G])$ is finitely generated for a finite group $G$ \cite{CarterLower}.
\end{remark}

We have now reduced the computation of the involution on $H_2(M;K_{-1}(\Z[G]))$ to the computation of the involution on $H_2(T^2;K_{-1}(\Z[G]))$ but this is just the involution on $K_{-1}(\Z[G])$.

We may now prove the following.
\begin{prop}\label{prop: involution on H_1(X;Wh)}
Suppose $G$ is a finite cyclic group of order at least $5$.
The involution on $H_1(X;\op{Wh}_{K}(F))_{(0)}$ has a $-1$-eigenspace.
It has a $1$-eigenspace if and only if $H_2(X;\Q)\neq0$ and there are distinct prime factors $p_i$ and $p_j$ of $\abs{G}$ such that $p_i$ has odd order in $(\Z/p_j)^\times$.
\end{prop}
\begin{proof}
By our assumption on the order of $G$, the Whitehead group is infinite.
By \cite{BakInvolution}, the involution on $\op{Wh}_1(G)$ is multiplication by $-1$.
So $H_0(X;\op{Wh}_1(G))_{(0)}$ is nontrivial and the involution is multiplication by $-1$.

The statement on $1$-eigenspaces follows from Proposition \ref{prop: H_2(X;K)} and Corollary \ref{cor: K_-1 involution}.
\end{proof}

Proposition \ref{prop: involution on H_1(X;Wh)} and Proposition \ref{prop: equiv homology suffices} prove Theorem \ref{thm: Main}.
\appendix
\section{The Involution on $K_{-1}(\Z[G])$}\label{section: negative K theory involution}

\subsection{Involutions on Spectra}
It is well-known that there are involutions on the $K$-theory spectra of group rings (and more generally of rings with involution).
Let $K(R[G])$ denote the connective $K$-theory spectrum of the group ring $R[G]$.
By regarding this as a space via Quillen's $+$-construction, an involution is given by the involution $\GL(R[G])\rightarrow\GL(R[G])$ sending a matrix to the inverse of its conjugate transpose.
Alternatively, one can also consider $K(R[G])$ as the $K$-theory of the symmetric monoidal category of finitely generated free $R$-modules.
Then, an involution is induced by the contravariant functor sending a module to its dual.

\begin{remark}
These define the same involution on connective $K$-theory but, on $K_1(R[G])$, it is the negative of the involution considered in \cite{MilnorWhitehead}.
\end{remark}

These involutions extend to involutions on non-connective $K$-theory spectra in the following sense.
Let $K^{-\infty}(R[G])$ denote the non-connective $K$-theory spectrum.
Then there is an involution on $K^{-\infty}(R[G])$ such that $K(R[G])\rightarrow K^{-\infty}(R[G])$ is a map of spectra with involution.

To be more explicit, one may consider, for instance, the Pedersen--Weibel model for $K^{-\infty}(R[G])$ \cite{PedersenWeibel}.
They consider additive categories $\mc{C}_{\R^n}(R[G])$ of finitely generated free $R[G]$-modules locally finitely indexed by points in $\R^n$.
Then, $K^{-\infty}(R[G])$ is defined to be an $\Omega$-spectrum with $n$-th space $K(\mc{C}_{\R^n}(R[G]))$.
One can define a contravariant functor on $\mc{C}_{\R^n}(R[G])$ which dualizes each module and preserves the coordinate in $\R^n$.
This makes $K^{-\infty}(R[G])$ into a spectrum with involution in the sense that it is an $\Omega$-spectrum whose spaces have involution and whose structure maps respect the involution.

\subsection{Dual Representations, $K_0$ and $K_1$}
If $x=\sum a_ig_i\in R[G]$, let $\overline{x}:=\sum a_ig_i^{-1}$.

\begin{definition}\label{def: tau_0 tau_1}
Let $P$ be a finitely generated projective $R[G]$-module.
Define the dual to be $P^*:=\Hom_{R[G]}(P,R[G])$ where, for $g\in G$, $x\in P$ and $f\in P^*$,
\[
(g\cdot f)(x)=f(x)\cdot g^{-1}.
\]
Define $\tau_0:K_0(R[G])\rightarrow K_0(R[G])$ by $[P]\mapsto[P^*]$.

Let $A=(a_{ij})$ be a matrix with coefficients in $R[G]$.
Define $A^*:=(\overline{a_{ji}})$ and $\tau_1:K_1(R[G])\rightarrow K_1(R[G])$ by $[A]\mapsto-\left[A^*\right]$.
\end{definition}

We note that $P^*$ is isomorphic as an $R[G]$-module to $\Hom_R(P,R)$ with the action defined by $(g\cdot \phii)(x)=\phii\left(g^{-1}\cdot x\right)$ for $\phii\in\Hom_R(P,R)$.
Indeed, if $f(x)=\sum_{g\in G}a_{g,x}g$, the map $\psi:P^*\rightarrow\Hom_R(P,R)$ sending $f$ to $\psi(f)(x)=a_{1,x}$ defines an isomorphism.

\begin{prop}\label{prop: K_0 in K_1 involution}
Let $\Phi:K_0(R[G])\rightarrow K_1(R[G\times\Z])$ be the homomorphism sending $[P]$ to $[te+(1-e)]$ where $t$ is a generator of $\Z$ and $e:R[G]^n\rightarrow R[G]^n$ is an idempotent matrix corresponding to the projective module $P$.
The following diagram is commutative.
\[
\begin{tikzpicture}[scale=2]
\node (A) at (0,1) {$K_0(R[G])$};\node (B) at (2,1) {$K_1(R[G\times\Z])$};
\node (C) at (0,0) {$K_0(R[G])$};\node (D) at (2,0) {$K_1(R[G\times\Z])$};
\path[->] (A) edge node[above]{$\Phi$} (B) (A) edge node[left]{$\tau_0$} (C) (B) edge node[left]{$\tau_1$} (D) (C) edge node[above]{$\Phi$} (D);
\end{tikzpicture}
\]
\end{prop}
\begin{proof}
The idempotent corresponding to $P^*$ is $e^*$ so
\[
\Phi\circ\tau_0([P])=\Phi\left(\left[P^*\right]\right)=\left[te^*+\left(1-e^*\right)\right].
\]
On the other hand,
\[
\tau_1\circ\Phi([P])=-\left[t^{-1}e^*+\left(1-e^*\right)\right]
\]
so $\Phi\circ\tau_0([P])=\tau_1\circ\Phi([P])$.
\end{proof}

\subsection{$K_{-1}$ and Localization Sequences}
In order to compute negative $K$-groups of group rings, localization sequences are very useful.
These sequences are obtained from a homotopy cartesian diagram of nonconnective $K$-theory spectra (see, for instance, \cite[V.7]{KBook}).
In our case, the maps of spectra are induced by maps of coefficient rings of group rings.
So, the maps in the sequences below will respect the involution.

\subsubsection{Carter's Sequence}
\begin{definition}\label{def: H_S}
Let $S$ be a central multiplicative subset of a ring $A$.
Define the category $\mathbf{H}_S(A)$ to be the $S$-torsion $A$ modules $M$ which have a finite length resolution of finitely generated projective $A$-modules.
\end{definition}

Let $S\subseteq\Z$ be a multiplicative subset generated by a set of primes and let $\langle p\rangle$ denote the multiplicative subset generated by $p$.
There is an equivalence of categories
\[
\mathbf{H}_S(\Z[G])\simeq\prod_{p\in S}\mathbf{H}_{\langle p\rangle}\left(\Z_p[G]\right)
\]
when $G$ is noetherian group.
This equivalence is given by sending an $S$-torsion $\Z[G]$-module to its $p$-primary parts.

Recall that, for a ring $A$, $K_{-1}(A)$ is defined to be the cokernel of $K_0(A[t])\oplus K_0\left(A\left[t^{-1}\right]\right)\rightarrow K_0\left(A\left[t,t^{-1}\right]\right)$.
Moreover, the map $K_0\left(A\left[t,t^{-1}\right]\right)\rightarrow K_{-1}(A)$ naturally splits so we may regard $K_{-1}(A)$ as a subgroup of $K_0\left(A\left[t,t^{-1}\right]\right)$.
Carter \cite{CarterLocalization} provides a resolution of free abelian groups computing $K_{-1}\left(\Z[G]\right)$ when $G$ is finite of order $n$.
\[
0\rightarrow K_0(\Z)\rightarrow K_0(\Q[G])\oplus\bigoplus_{p|n}K_0\left(\Z_p[G]\right)\rightarrow\bigoplus_{p|n}K_0\left(\Q_p[G]\right)\xrightarrow{\partial} K_{-1}(\Z[G])\rightarrow0
\]
The map $K_0\left(\Q_p[G]\right)\rightarrow K_{-1}(\Z[G])$ is defined using a connecting homomorphism $\partial:K_1\left(\Q_p[G\times\Z]\right)\rightarrow K_0(\Z[G\times\Z])$.

This connecting homomorphism $\partial$ is defined to be a composite
\[
K_1\left(\Q_p[G\times\Z]\right)\rightarrow K_0\mathbf{H}_{\langle p\rangle}\left(\Z_p[G\times\Z]\right)\rightarrow K_0\mathbf{H}_{\langle p\rangle}\left(\Z[G\times\Z]\right)\rightarrow K_0(\Z[G]).
\]
Suppose $A\in\GL_n(\Q_p[G\times\Z])$ is a matrix representing an element of $K_1\left(\Q_p[G\times\Z]\right)$.
There is an $r\ge0$ such that $p^r A$ has coefficients in $\Z_p[G\times\Z]$.
The first map sends $A$ to $\left[\op{coker}\left(p^r A\right)\right]-\left[\op{coker}\left(p^rI_n\right)\right]$.
The second map sends a $p$-primary group regarded as a module over $\Z_p[G\times\Z]$ to the same group regarded as a module over $\Z[G\times\Z]$.
The third map sends an $S$-torsion module with a finite length resolution to the Euler characteristic of the resolution.

Note that
\[
\Z_p[G\times\Z]^n\xrightarrow{p^r A}\Z_p[G\times\Z]^n\rightarrow\op{coker}\left(p^r A\right)
\]
is a projective resolution of $\Z_p[G\times\Z]$-modules.
The argument in the proof of \cite[Lemma 2.3]{CarterLocalization} shows there is a projective resolution of $\Z[G\times\Z]$-modules
\[
F\rightarrow\Z[G\times\Z]^m\rightarrow\op{coker}\left(p^r A\right).
\]
One can similarly describe the $\op{coker}\left(p^rI_n\right)$ term and conclude that
\[
\partial[A]=\left[\Z[G\times\Z]^m\right]-[F].
\]

One can give $K_{-1}(\Z[G])$ and involution by restricting the involution on $K_0(\Z[G\times\Z])$.
The following result shows that the Carter sequence respects this involution.

\begin{prop}\label{prop: connecting map involution}
The following diagrams commute.
\[
\begin{tikzpicture}[scale=2]
\node (A) at (0,1) {$K_1\left(\Q_p[G\times\Z]\right)$};\node (B) at (2,1) {$K_0\left(\Z[G\times\Z]\right)$};
\node (C) at (0,0) {$K_1\left(\Q_p[G\times\Z]\right)$};\node (D) at (2,0) {$K_0\left(\Z[G\times\Z]\right)$};
\path[->] (A) edge node[above]{$\partial$} (B) (A) edge node[left]{$\tau_1$} (C) (B) edge node[left]{$\tau_0$} (D) (C) edge node[above]{$\partial$} (D);
\end{tikzpicture}
\hspace{1cm}
\begin{tikzpicture}[scale=2]
\node (A) at (0,1) {$K_0\left(\Q_p[G]\right)$};\node (B) at (2,1) {$K_{-1}\left(\Z[G]\right)$};
\node (C) at (0,0) {$K_0\left(\Q_p[G]\right)$};\node (D) at (2,0) {$K_{-1}\left(\Z[G]\right)$};
\path[->] (A) edge node[above]{$\partial$} (B) (A) edge node[left]{$\tau_0$} (C) (B) edge node[left]{$\tau_{-1}$} (D) (C) edge node[above]{$\partial$} (D);
\end{tikzpicture}
\]
\end{prop}
\begin{proof}
The second diagram follows from the first and Proposition \ref{prop: K_0 in K_1 involution}.

We show that the first diagram commutes.
Let $[A]\in K_1\left(\Q_p[G\times\Z]\right)$ and define $M:=\op{coker}\left(p^r A\right)$
Let
\begin{equation}\label{eq: Z res for M}
0\rightarrow F\rightarrow\Z[G\times\Z]^m\rightarrow M\rightarrow0
\end{equation}
be as above.
It follows immediately that
\[
\tau_0\circ\partial [A]=\left[\Z[G\times\Z]^m\right]-\left[F^*\right].
\]

Instead of evaluating $\partial\circ\tau_1[A]$, it will be slightly easier to evaluate $\partial\circ(-\tau_1)[A]$.
There is an exact sequence
\[
0\rightarrow\Hom_{\Z_p}\left(M,\Z_p\right)\rightarrow\Z_p[G\times\Z]^n\xrightarrow{A^*}\Z_p[G\times\Z]^n\rightarrow\Ext^1_{\Z_p}\left(M,\Z_p\right)\rightarrow0.
\]
The term $\Hom_{\Z_p}\left(M,\Z_p\right)$ vanishes since $M$ is torsion.
So to compute $\partial\circ(-\tau_1)[A]$ we need a projective $\Z[G\times\Z]$-resolution of $\Ext^1_{\Z_p}(M,\Z_p)$.

Dualizing (\ref{eq: Z res for M}) above gives a projective $\Z[G\times\Z]$-resolution
\[
0\rightarrow\Z[G\times\Z]^m\rightarrow F^*\rightarrow\Ext^1_{\Z}(M,\Z)\rightarrow0
\]
Since $\Ext^1_{\Z_p}\left(M,\Z_p\right)\cong\Ext^1_{\Z}\left(M,\Z_p\right)$ it suffices to show that $\Ext^1_{\Z}\left(M,\Z_p\right)\cong\Ext^1_{\Z}(M,\Z)$.
This isomorphism follows by considering the injective resolutions
\begin{align*}
0\rightarrow\Z\rightarrow&\Q\rightarrow\Q/\Z\rightarrow0\\
0\rightarrow\Z_p\rightarrow&\Q_p\rightarrow\Z\left[\frac{1}{p}\right]/\Z\rightarrow0
\end{align*}
and recalling that $M$ is $p$-primary.
\end{proof}

\subsubsection{The Madsen-Rothenberg Sequence}
In \cite{MadsenRothenbergEquivariantAut}, Madsen and Rothenberg regard the functor $K\left(R[-]\right)$ as a Mackey functor.
It follows that $K_n\left(R[G]\right)$ has an action of the Burnside ring $A(G)$.
Let $q(G,0)\subseteq A(G)$ denote the ideal generated by the virtual finite $G$-sets whose $G$-fixed point set has order $0$.
If $\mc{M}$ is a Mackey functor, then localization at this ideal can be described as follows.
\begin{equation}\label{eq: localize burnside}
\mc{M}\left(G/G\right)_{q(G,0)}=\ker\left(\mc{M}(G/G)_{(0)}\rightarrow\bigoplus_{(H)}\mc{M}(G/H)_{(0)}\right)
\end{equation}
Here, the $H$ on the right hand side varies over conjugacy classes of proper subgroups of $G$.
Heuristically, this localization is isolating the part of $\mc{M}(G/G)_{(0)}$ which does not come from a proper subgroup.

Let $G=\Z/m\Z$ be finite cyclic.
For a subgroup $H$, the composite
\[
\mc{M}(G/H)_{(0)}\rightarrow\mc{M}(G/G)_{(0)}\rightarrow\mc{M}(G/H)_{(0)}
\]
is multiplication by the index so it is a vector space isomorphism.

Madsen--Rothenberg claim that localizing the Carter sequence at $q(G,0)$ gives the following short exact sequence.
\[
0\rightarrow K_0\left(\Q\left(\zeta_m\right)\right)_{(0)}\rightarrow \bigoplus_{p|m}K_0\left(\Q_p\otimes_{\Q}\Q\left(\zeta_m\right)\right)_{(0)}\rightarrow K_{-1}\left(\Z\left[G\right]\right)_{q(0,2)}\rightarrow0
\]
Indeed, writing $\Q[G]$ as a product of cyclotomic fields, we see that only the summand $K_0\left(\Q_p\otimes\Q\left(\zeta_m\right)\right)_{(0)}$ is in the kernel above.
Additionally, if we write $m=p^r m_p$ where $p$ does not divide $m_p$ then
\begin{align*}
K_0\left(\Z_p[G]\right)&\cong K_0\left(\Z_p\left[\Z/p^r\Z\right]\left[\Z/m_p\Z\right]\right)\cong K_0\left(\F_p\left[\Z/p^r\Z\right]\left[\Z/m_p\Z\right]\right)\\
&\cong K_0\left(\F_p[x]\left[\Z/m_p\Z\right]/\left(x^{p^r}-1\right)\right)\cong K_0\left(\F_p\left[\Z/m_p\Z\right]\right)\cong K_0\left(\Z_p\left[\Z/m_p\Z\right]\right).
\end{align*}
The second and last isomorphisms follow from the fact that $(p)$ is a complete ideal in $\Z_p$.
The fourth isomorphism follows from the fact that the ideal $(x-1)$ is nilpotent.
Therefore, $K_0\left(\Z_p\left[G\right]\right)_{q(G,0)}=0$.

The action on the middle term is more complicated.
We will need the following lemma.
\begin{lemma}\label{lem: product of fields}
Suppose $K/\Q$ is a finite Galois extension.
Then $\Q_p\otimes_{\Q}K$ is a product of isomorphic fields.
\end{lemma}
\begin{proof}
We may write $K=\Q[x]/f(x)$ and $\Q_p\otimes_{\Q}K=\Q_p[x]/f(x)=\Q_p[x]/f_1(x)\cdots f_s(x)$ where $f(x)=f_1(x)\cdots f_s(x)$ is a factorization into irreducible polynomials in $\Q_p$.
So
\[
\Q_p\otimes_{\Q}K\cong\prod_{i=1}^s\Q_p[x]/f_i(x)
\]
where each $\Q_p[x]/f_i(x)$ is a field.
The Galois group of $K/\Q$ acts transitively on the roots of $f$ so there is an automorphism $\sigma$ sending a root of $f_a(x)$ to a root of $f_b(x)$.
This induces a ring automomorphism of $\Q_p\otimes_{\Q}K$.

Consider the composite
\[
\Q_p[x]/f_a(x)\rightarrow\prod_{i=1}^s\Q_p[x]/f_i(x)\xrightarrow{\sigma}\prod_{i=1}^s\Q_p[x]/f_i(x)\rightarrow\Q_p[x]/f_b(x).
\]
The first map sends an element $g(x)$ to the element which is $g(x)$ in the coordinate indexed my $a$ and $0$ elsewhere.
This is a non-unital ring homomorphism.
The composite is a nonzero field homomorphism so it is injective.
Similarly, $\sigma^{-1}$ gives a nonzero field homomorphism going the other way.
Since these are finite dimensional $\Q_p$-vector spaces, we see that $\Q_p[x]/f_a(x)\cong\Q_p[x]/f_b(x)$.
\end{proof}

In our case, we are interested in $K=\Q(\zeta)$.

\begin{prop}\label{prop: product of cyclotomic fields}
Let $\zeta$ be an $m$-th root of unity and let $p$ be a prime divisor of $m$.
Write $m=p^rm_p$ where $p$ does not divide $m_p$.
There is an isomorphism $\Q_p\otimes_{\Q}\Q(\zeta)\cong\prod_{i=1}^s\Q_p(\zeta)$ where $s$ is the index of $p$ in $(\Z/m_p)^\times$.
\end{prop}
\begin{proof}
Let $t$ denote the order of $p$ in $(\Z/m_p)^\times$.
The degree of the extension $\Q_p(\zeta)/\Q_p$ is $t(p-1)p^{r-1}$ (see \cite[IV.4]{SerreLocalFields}) and the degree of the extension $\Q(\zeta)$ is $\abs{(\Z/m_p)^\times}(p-1)p^{r-1}$.
The result follows from Lemma \ref{lem: product of fields}.
\end{proof}

\subsubsection{Involutions on $K_0\left(\Q_p[G]\right)$}
An analysis of the involution on $K_0\left(\Q_p[G]\right)$ follows easily from \cite[12.4]{SerreRep}.
Let $K$ be a field of characteristic $0$ and $G$ a finite group with order $m$.
Define $L:=K\left(\zeta_m\right)$ where $\zeta_m$ is a primitive $m$-th root of unity then $\Gal(L/K)\subseteq\left(\Z/m\Z\right)^\times$.
Let $\Gamma_K$ denote the image of the Galois group in $\left(\Z/m\Z\right)^\times$.
Two elements $s$ and $s'$ of $G$ are $\Gamma_K$ conjugate if there is a $t\in\Gamma_k$ such that $s^t$ and $s'$ are conjugate in $G$.
The following is \cite[12.4 Corollary 1]{SerreRep}.
\begin{cor}\label{cor: Serre Rep}
A class function $f:G\rightarrow K$ belongs to $K\otimes_{\Z} R_K(G)$ if and only if it is constant on $\Gamma_K$-classes of $G$.
\end{cor}

\begin{lemma}\label{lem: rep ring free or trivial}
Let $G$ be an odd order abelian group.
Then $\Z[\Z/2]$-module $R_K(G)/\langle\op{triv}\rangle$ is either free or a free abelian group with a trivial involution.
In the first case, the set of nontrivial irreducible $G$-representations over $K$ form a free $\Z/2$-set.
\end{lemma}
\begin{proof}
If $-1\in\Gamma_K$ then all characters $\chi$ satisfy $\chi(g)=\chi(g^{-1})$.
Suppose $\-1\notin\Gamma_K$.
Since we have assumed $\abs{G}$ is odd, there is no nontrivial $g\in G$ such that $g=g^{-1}$ so $K\otimes_{\Z}R_K(G)/\langle\op{triv}\rangle$ is a free $K[\Z/2]$-module.
Also, $R_K(G)$ is a finitely generated $\Z[\Z/2]$-module which is obtained by linearizing the $\Z/2$-set of irreducible $G$-representations over $K$.
It follows that the set of nontrivial irreducible representations must be a free $\Z/2$-set.
\end{proof}

Let $G=\Z/m$ where $m$ is odd and let $\zeta$ be a primitive $m$-th root of unity as before.
In this case, $\Gamma_{\Q_p}=\Gal(\Q_p(\zeta)/\Q_p)\le(\Z/m)^\times$.
The following lemma records our knowledge of the Galois group $\Gal(\Q_p(\zeta)/\Q_p)$.

\begin{lemma}\label{lem: Gal(Q_p(zeta))}
Suppose $p$ divides $m$.
The Galois group $\Gal(\Q_p(\zeta)/\Q_p)\le(\Z/m)^\times$ contains $-1$ if and only if, for each prime factor $p_j$ of $m$ not equal to $p$, the group $\langle p\rangle\le(\Z/p_j)^\times$ contains $-1$.
\end{lemma}
\begin{proof}
Factor $m=p_1^{r_1}p_2^{r_2}\cdots p_k^{r_k}$.
There is an injection of Galois groups
\[
\Gal(\Q_p(\zeta_m)/\Q_p)\rightarrow\Gal(\Q_p(\zeta_{p_1^{r_1}})/\Q_p)\times\cdots\times\Gal(\Q_p(\zeta_{p_k^{r_k}})/\Q_p)
\]
such that composition with each projection on the right hand side is a surjection.
Under the isomorphism
\[
\left(\Z/m\right)^\times\cong\left(\Z/p_1^{r_1}\right)^\times\times\cdots\times\left(\Z/p_k^{r_k}\right)^\times
\]
$-1$ is mapped to $(-1,-1,\cdots,-1)$.
For $p_j=p$, $\Gal(\Q_p(\zeta_{p^r})/\Q_p)\cong(\Z/p^r)^\times$ so $-1$ is always in the image of this component.

Assume $p_j\neq p$.
To prove the lemma, it suffices to show that $-1$ is in $\Gal(\Q_p(\zeta_{p_j^{r_j}})/\Q_p)\le(\Z/p_j^{r_j})^\times$ if and only if $\langle p\rangle\le(\Z/p_j)^\times$ contains $-1$.
This group $\Gal(\Q_p(\zeta_{p_j^{r_j}})/\Q_p)$ is cyclic with order equal to the order of $p$ in $(\Z/p_j^{r_j})^\times$ \cite[IV.4]{SerreLocalFields}.
It is straightforward to check that $p$ has even order in $(\Z/p_j^{r_j})^\times$ if and only if it has even order in $(\Z/p_j)^\times$.
\end{proof}

The abelian group $K_0(\Q_p\otimes_{\Q}\Q(\zeta))$ inherits an involution from the involution $[P]\mapsto[P^*]$ on $K_0(\Q_p[G])$.

\begin{cor}\label{cor: cyclotomic decomp free or trivial}
The $\Z[\Z/2]$-module $K_0(\Q_p\otimes_{\Q}\Q(\zeta))$ is free if and only if, for each prime factor $p_j$ of $m$, $p\neq p_j$, the order of $p$ in $(\Z/p_j)^\times$ is odd.
Otherwise the involution is trivial.
\end{cor}

\begin{cor}\label{cor: K_-1 involution}
The involution on $K_{-1}(\Z[G])_{(0)}$ has a $-1$-eigenspace if and only if there are distinct prime factors $p_i,p_j$ of $\abs{G}$ such that the order of $p_i$ in $(\Z/p_j)^\times$ is odd.
Otherwise the involution is trivial.
\end{cor}

\bibliographystyle{alpha}
\bibliography{InvolutionOnWhitehead.bib}

\begin{thebibliography}{Car80b}

\bibitem[Bak77]{BakInvolution}
Anthony Bak.
\newblock The involution on {W}hitehead torsion.
\newblock {\em General Topology and Appl.}, 7(2):201--206, 1977.

\bibitem[BH78]{BrowderHsiangProblem}
W.~Browder and W.~C. Hsiang.
\newblock Some problems on homotopy theory manifolds and transformation groups.
\newblock In {\em Algebraic and geometric topology ({P}roc. {S}ympos. {P}ure
  {M}ath., {S}tanford {U}niv., {S}tanford, {C}alif., 1976), {P}art 2}, volume
  XXXII of {\em Proc. Sympos. Pure Math.}, pages 251--267. Amer. Math. Soc.,
  Providence, RI, 1978.

\bibitem[Car80a]{CarterLocalization}
David~W. Carter.
\newblock Localization in lower algebraic {$K$}-theory.
\newblock {\em Comm. Algebra}, 8(7):603--622, 1980.

\bibitem[Car80b]{CarterLower}
David~W. Carter.
\newblock Lower {$K$}-theory of finite groups.
\newblock {\em Comm. Algebra}, 8(20):1927--1937, 1980.

\bibitem[DL98]{DavisLuckEquivariant}
James~F. Davis and Wolfgang L\"{u}ck.
\newblock Spaces over a category and assembly maps in isomorphism conjectures
  in {$K$}- and {$L$}-theory.
\newblock {\em $K$-Theory}, 15(3):201--252, 1998.

\bibitem[FLS18]{FarrellLuckSteimle}
Tom Farrell, Wolfgang L\"{u}ck, and Wolfgang Steimle.
\newblock Approximately fibering a manifold over an aspherical one.
\newblock {\em Math. Ann.}, 370(1-2):669--726, 2018.

\bibitem[JK18]{JahrenKwasik}
Bj\o~rn Jahren and S\l~awomir Kwasik.
\newblock Whitehead torsion of inertial {$h$}-cobordisms.
\newblock {\em Topology Appl.}, 249:150--159, 2018.

\bibitem[LR78]{LashofRothenberg}
R.~Lashof and M.~Rothenberg.
\newblock {$G$}-smoothing theory.
\newblock In {\em Algebraic and geometric topology ({P}roc. {S}ympos. {P}ure
  {M}ath., {S}tanford {U}niv., {S}tanford, {C}alif., 1976), {P}art 1}, volume
  XXXII of {\em Proc. Sympos. Pure Math.}, pages 211--266. Amer. Math. Soc.,
  Providence, RI, 1978.

\bibitem[LR05]{LuckReichSurvey}
Wolfgang L\"uck and Holger Reich.
\newblock The {B}aum-{C}onnes and the {F}arrell-{J}ones conjectures in {$K$}-
  and {$L$}-theory.
\newblock In {\em Handbook of {$K$}-theory. {V}ol. 1, 2}, pages 703--842.
  Springer, Berlin, 2005.

\bibitem[Luc89]{LuckTransformation}
Wolfgang Luck.
\newblock {\em Transformation groups and algebraic {$K$}-theory}, volume 1408
  of {\em Lecture Notes in Mathematics}.
\newblock Springer-Verlag, Berlin, 1989.
\newblock Mathematica Gottingensis.

\bibitem[Mil66]{MilnorWhitehead}
J.~Milnor.
\newblock Whitehead torsion.
\newblock {\em Bull. Amer. Math. Soc.}, 72:358--426, 1966.

\bibitem[MR88]{MadsenRothenbergEquivariantAut}
Ib~Madsen and Mel Rothenberg.
\newblock On the homotopy theory of equivariant automorphism groups.
\newblock {\em Invent. Math.}, 94(3):623--637, 1988.

\bibitem[PW85]{PedersenWeibel}
Erik~K. Pedersen and Charles~A. Weibel.
\newblock A nonconnective delooping of algebraic {$K$}-theory.
\newblock In {\em Algebraic and geometric topology ({N}ew {B}runswick,
  {N}.{J}., 1983)}, volume 1126 of {\em Lecture Notes in Math.}, pages
  166--181. Springer, Berlin, 1985.

\bibitem[Qui82]{QuinnEnds2}
Frank Quinn.
\newblock Ends of maps. {II}.
\newblock {\em Invent. Math.}, 68(3):353--424, 1982.

\bibitem[Sch79]{SchultzExotic}
Reinhard Schultz.
\newblock Spherelike {$G$}-manifolds with exotic equivariant tangent bundles.
\newblock In {\em Studies in algebraic topology}, volume~5 of {\em Adv. Math.
  Suppl. Stud.}, pages 1--38. Academic Press, New York-London, 1979.

\bibitem[Ser77]{SerreRep}
Jean-Pierre Serre.
\newblock {\em Linear representations of finite groups.}
\newblock Springer-Verlag, New York-Heidelberg,,, french edition, 1977.

\bibitem[Ser79]{SerreLocalFields}
Jean-Pierre Serre.
\newblock {\em Local fields}, volume~67 of {\em Graduate Texts in Mathematics}.
\newblock Springer-Verlag, New York-Berlin, 1979.
\newblock Translated from the French by Marvin Jay Greenberg.

\bibitem[Ste88]{Steinberger}
Mark Steinberger.
\newblock The equivariant topological {$s$}-cobordism theorem.
\newblock {\em Invent. Math.}, 91(1):61--104, 1988.

\bibitem[SW85]{SteinbergerWest}
Mark Steinberger and James West.
\newblock Equivariant {$h$}-cobordisms and finiteness obstructions.
\newblock {\em Bull. Amer. Math. Soc. (N.S.)}, 12(2):217--220, 1985.

\bibitem[Vog85]{VogellInvolution}
Wolrad Vogell.
\newblock The involution in the algebraic {$K$}-theory of spaces.
\newblock In {\em Algebraic and geometric topology ({N}ew {B}runswick,
  {N}.{J}., 1983)}, volume 1126 of {\em Lecture Notes in Math.}, pages
  277--317. Springer, Berlin, 1985.

\bibitem[Wan23]{WangChern}
Oliver~H. Wang.
\newblock Chern class obstructions to smooth equivariant rigidity.
\newblock 2023.

\bibitem[Wei13]{KBook}
Charles~A. Weibel.
\newblock {\em The {$K$}-book}, volume 145 of {\em Graduate Studies in
  Mathematics}.
\newblock American Mathematical Society, Providence, RI, 2013.
\newblock An introduction to algebraic $K$-theory.

\end{thebibliography}
\end{document}